\documentclass[a4paper]{amsart}

\usepackage{mathrsfs,amsmath} 
\usepackage{amssymb}

\newtheorem{theorem}{Theorem}
\newtheorem{corollary}{Corollary}

\newtheorem{lemma}{Lemma}
\newtheorem*{lemma*}{Lemma}

\theoremstyle{definition}
\newtheorem{definition}[theorem]{Definition}
\newtheorem{remark}{Remark}

\newcommand{\xii}{|\xi|}
\newcommand{\R}{{\mathbb{R}}}
\newcommand{\N}{{\mathbb{N}}}

\newcommand{\BMO}{{\mathrm{BMO}}}
\newcommand{\PV}{{\mathrm{P.V.}}}
\renewcommand{\div}{{\mathrm{div}}}

\def\<#1\>{\left\langle#1\right\rangle }

\title[A class of $L^1(\R^n)-H^1(\R^n)$ multipliers]{A class of multipliers \\ from $L^1(\R^n)$ to the Hardy space $H^1(\R^n)$}

\author[M. D'Abbicco]{Marcello D'Abbicco}

\address{Department of Mathematics, University of Bari, Via E. Orabona 4, 70125 BARI - ITALY}

\email{marcello.dabbicco@uniba.it}



\begin{document}

\begin{abstract}
In this note, we show a condition on the derivatives of a multiplier $m\in L^\infty$ that guarantees that $m$ is a multiplier from $L^1$ to the Hardy Space $H^1(\R^n)$.
\end{abstract}

\maketitle

\section{Introduction}

The theory of multipliers arise in several fields of mathematical analysis, in particular in the study of partial differential equations in $\R^n$, due to the possibility to use the Fourier transform to reduce a partial differential equation with constant coefficients to an algebraic equation depending on the frequency as a parameter. For instance, if $E(t)\in\mathcal S'(\R^n)$ is the fundamental solution to the initial value problem
\[ \begin{cases}
\partial_t E + \sum_{|\alpha|\leq m} c_\alpha\partial_x^\alpha E =0, \qquad t>0,  \\
E(0)=\delta
\end{cases} \]
after the application of the Fourier transform with respect to the space variable $\mathscr{F}$, it reads as
\[ \begin{cases}
\partial_t \mathscr{F} E + p(\xi)\,\mathscr{F} E =0, \qquad t>0,  \\
\mathscr{F}E(0)=1,
\end{cases} \]
where $p(\xi)=\sum_{|\alpha|\leq m} c_\alpha(i\xi)^\alpha$ (for instance, $p(\xi)=\xii^2$ for the heat equation and $p(\xi)=i\xii^2$ for the Schr\"odinger equation). Then $m(t,\xi)=e^{-tp(\xi)}$ is the multiplier associated to the operator $E(t)\ast$ that maps an initial datum $u_0$ into the solution of the initial value problem $E(t)\ast u_0$.

In a similar way, several other problems may be reduced to the study of (time-dependent) multipliers. One may consider initial value problems with higher order derivatives in time, as wave models, or problems with time-dependent coefficients and/or with more general operators than ordinary derivatives.

It often happens that one may easily obtain properties of the multiplier $m(t,\xi)$, even in absence of information on the properties of $E(t)=\mathscr{F}^{-1}m$ itself.

Fourier analysis provide several tools to show that a multiplier $m\in L^\infty(\R^n)$ is in $M(L^p,L^p)$ for some $p\in(1,\infty)$, that is, that $(\mathscr{F}^{-1}m)\ast$ is a linear bounded operator that maps $L^p$ into itself. In particular, as a consequence of the Mikhlin-H\"ormander multiplier theorem (see~\cite{Mikhlin56} and~\cite[Theorem 2.5]{Hormander-1960}), if $\xii^{|\alpha|}\partial_\xi^\alpha m(\xi)$ are bounded for any $|\alpha|\leq1+n/2$, then $m\in M(L^p,L^p)$ for any $p\in(1,\infty)$. The result remains valid for $M(H^1,H^1)$, where $H^1(\R^n)$ is the Hardy space, and for $M(\BMO,\BMO)$, where $\BMO$ is the dual space of $H^1$, the space of functions with bounded mean oscillations (see~\cite[Theorem E]{Miyachi-1981}; see also~\cite[Theorems 4.6 and 4.7]{CALDERON1977101} and~\cite[Theorems 1 and 2]{Miyachi-1980}).

On the other hand, a multiplier is in $M(L^1,L^1)$ if, and only if, its inverse Fourier transform is a bounded measure~\cite[Theorem 1.4]{Hormander-1960}. Moreover, a multiplier is in $M(L^1,H^1)$ if, and only if, its inverse Fourier transform belongs to $H^1(\R^n)$ (see~\cite[Theorem 3.3]{Miyachi-1981}).

An example of interest is the fundamental solution to the wave equation, that is, $m(\xi)=\xii^{-1}\sin \xii$. Then $m\in M(L^1,L^1)$ in space dimension~$n=3$, since $\mathscr{F}^{-1}m$ is the normalized measure of the unit sphere, and $m\in M(L^p,L^p)$ for $(n-1)|1/p-1/2|\leq1$ if $n\geq4$ (see~\cite{Peral-1980}).

Multipliers in $M(L^1,H^1)$ are of interest in the study of partial differential equations. For instance, if a function is the product of a multiplier in $M(L^1,H^1)$ and a multiplier in $M(H^1,H^1)$ (which may often be easily checked using Mikhlin-H\"ormander theorem), then the function is a multiplier in $M(L^1,H^1)$ and, in particular, it is in $M(L^1,L^1)$. Examples of multipliers that are in $M(H^1,H^1)$ but not in $M(L^1,L^1)$ are $\Phi\, e^{i\omega}$, where $\omega$ is homogeneous of order $1$ and $\Phi\in S^{-\frac{n-1}2}$ (see~\cite[IX, \textsection4]{stein1993harmonic}) or where $\omega$ is homogeneous of order $\sigma\neq1$ and $\Phi\in S^{-\frac{n}2}$ (see~\cite[Theorem 1]{D'Abbicco-2025}).

Multipliers in $M(L^1,H^1)$ are also of interest to prove $L^1-L^1$ estimates for partial differential equations and systems where the Riesz transforms appear in the fundamental solution, as for the system of dispersive elastic waves (or other models with the operator $\nabla \div$).

Our main result is the following.
\begin{theorem}\label{thm:main}
Let $n\geq2$, $m\in\mathcal S'(\R^n)$, be such that its derivatives in distributional sense of order $n-1$, $n$ and $n+1$ are functions, and they verify
\begin{equation}\label{eq:mder}
|\partial_\xi^\gamma m(\xi)|\leq C\xii^{\varepsilon-|\gamma|}\<\xi\>^{-2\varepsilon},\qquad \gamma\in\N^n,\quad |\gamma|=n-1,n,n+1,
\end{equation}
for some~$C>0$ and~$\varepsilon>0$. Then there exists a finite family of coefficients $\{a_\alpha\}_{|\alpha|\leq n-2}$ such that
\begin{equation}\label{eq:mT0}
\mathscr{F}^{-1}m = f + \sum_{|\alpha|\leq n-2} a_\alpha \partial_x^\alpha \delta,
\end{equation}
with $f\in L^1(\R^n)\cap \mathcal C(\R^n\setminus\{0\})$ and $|f(x)|\leq C'|x|^{\varepsilon-n}\<x\>^{-2\varepsilon}$. Moreover, $f\in H^1(\R^n)$ if, and only if, $m(0)=a_0$ (that is, $(\mathscr{F}f)(0)=0$).
\end{theorem}
We postpone the proof of Theorem~\ref{thm:main} to~\textsection\ref{sec:proofmain}.

If we are interested in multipliers that may be bounded on $L^p$ for some $p\in[1,\infty]$, we may assume that $m\in M(L^2,L^2)=L^\infty(\R^n)$.
\begin{corollary}\label{cor:main}
Let $m\in L^\infty(\R^n)$, be such that~\eqref{eq:mder} holds. Then:
\begin{align}
\label{eq:RL}
& \exists \lim_{\xii\to\infty} m(\xi)= a_0,\\
\label{eq:deltaf}
& \mathscr{F}^{-1}m=a_0\delta+f,
\end{align}
with $f$ as in Theorem~\ref{thm:main}. In particular, $m\in M(L^1,L^1)$. Moreover, $m\in M(L^1,H^1)$ if, and only if, $m(0)=0=a_0$.
\end{corollary}
\begin{proof}
Since $m$ is bounded, we derive $a_\alpha=0$ for any $\alpha\neq0$, so that~\eqref{eq:deltaf} holds. By Riemann-Lebesgue theorem for $f$, we obtain
\[  \lim_{\xii\to\infty} m(\xi) = a_0+ \lim_{\xii\to\infty} \mathscr{F} f(\xi) = a_0, \]
that is, \eqref{eq:RL}.

By~\cite[Theorem 1.4]{Hormander-1960} (indeed, it is sufficient to apply Young theorem for bounded measures), $m\in M(L^1,L^1)$. If $m(0)=0=a_0$, then $m$ is the Fourier transform of $f\in H^1$; hence, by~\cite[Theorem 3.3]{Miyachi-1981}, $m\in M(L^1,H^1)$.

This concludes the proof.
\end{proof}
\begin{remark}
Corollary~\ref{cor:main} applies to pseudo-differential operators $m\in S^{-\varepsilon}$ depending on the variable $\xi$ only; in particular, $\mathscr{F}^{-1}m=f\in L^1$, and $f\in H^1(\R^n)$ if $m(0)=0$. In this sense, Corollary~\ref{thm:main} improves~\cite[Lemma 3.4]{D'Abbicco-2025}.
\end{remark}
The proof of Theorem~\ref{thm:main} is based on the following sufficient condition for a function to be in $H^1(\R^n)$.
\begin{lemma}\label{lem:H1}
Let $f\in L^p(\R^n)$ for some $p>1$, and assume that
\[ \lim_{|x|\to\infty} |x|^{n+\varepsilon}|f(x)|=0, \]
for some $\varepsilon>0$. Then $f\in H^1(\R^n)$ if, and only if, $f$ has zero integral.
\end{lemma}
\begin{remark}
The condition on the limit of $|x|^{n+\varepsilon}|f(x)|$ comes into play to prove that the function~$f$ is in $H^1$ (it is a sufficient, not necessary, condition, that arises in a natural way from the assumption~\eqref{eq:mder}). Indeed, there exist functions in $L^1\cap L^\infty$, with zero integral, that are not in $H^1$.
\end{remark}
The proof of Lemma~\ref{lem:H1} is based on the following.
\begin{lemma}\label{lem:H1large}
Let $f\in L^1(\R^n)$, with $f(x)=0$ in $B(0,M)=\{x\in \R^n: \ |x|\leq M\}$ and $|f(x)|\leq |x|^{-n-\varepsilon}$ for some $\varepsilon>0$. Assume that $f$ has zero integral. Then $f\in H^1(\R^n)$ and
\[ \|f\|_{H^1(\R^n)}\leq CM^{-\varepsilon}, \]
for some $C>0$, independent on $M$.
\end{lemma}
We postpone the proofs of Lemmas~\ref{lem:H1} and~\ref{lem:H1large} to~\textsection\ref{sec:H1}.

\section{Notation and tools from Fourier analysis}

In this paper, $H^1(\R^n)$ denotes the Hardy space on $\R^n$. There are several descriptions of this space and its norm, we focus on two of them.
\begin{definition}[see, for instance, \cite{Grafakos-2024}]\label{def:atom}
An $H^1$ atom is a function $a\in L^1(\R^n)$ such that:
\begin{itemize}
  \item $a$ is supported in a cube $Q$ (with sides parallel to the axes);
  \item $|a(x)|\leq |Q|^{-1}$ for all $x\in Q$;
  \item $a$ has zero integral.
\end{itemize} 
\end{definition}
Using the atomic decomposition, we may give one of the equivalent definitions of the space $H^1(\R^n)$.
\begin{definition}\label{def:atomic}
Assume that $f$ is a measurable function on $\R^n$ such that
\[ f(x) = \sum_{j=0}^\infty \lambda_j a_j(x), \quad \text{a.e., with}\qquad \sum_{j=0}^\infty |\lambda_j|<\infty, \]
where $a_j$ are $H^1$-atoms. Then we say that $f\in H^1(\R^n)$ and we define the norm
\begin{equation}\label{eq:normH1}
\|f\|_{H^1}=\inf \{\sum_{j=0}^\infty |\lambda_j|: \ f=\sum_{j=0}^\infty \lambda_j a_j(x), \ \text{where $a_j$ are $H^1$ atoms}\}.
\end{equation}
\end{definition}
We use the following notation for multi-indexes $\alpha=(\alpha_1,\ldots,\alpha_n)$ in $\N^n$ and related polynomials and derivatives:
\[ |\alpha|=\alpha_1+\ldots+\alpha_n, \qquad x^\alpha=x_1^{\alpha_1}\cdot\ldots\cdot x_n^{\alpha_n}, \qquad \partial_\xi^\alpha = \partial_{\xi_1}^{\alpha_1}\ldots \partial_{\xi_n}^{\alpha_n}. \]
In this paper, we use the following notation for the Fourier transform. For any function $f\in L^1(\R^n)$, we set
\[ \mathscr{F}f(\xi) = \int_{\R^n} e^{-ix\cdot\xi} f(x)\,dx. \]
Then $\mathscr{F}$ is an isomorphism on the Schwartz space $\mathcal S(\R^n)$, with inverse $\mathscr{F}^{-1}\varphi=(2\pi)^{-n}\mathscr{F}\check{\varphi}$, where $\check{\varphi}(x)=\varphi(-x)$.

The definition is extended to the space $\mathcal S'(\R^n)$ of tempered distribution, by setting
\[ \<\mathscr{F}T, \varphi\> = \<T, \mathscr{F}\varphi\>, \qquad \varphi\in\mathcal S(\R^n). \]
For any $T\in\mathcal S'(\R^n)$, we have
\[ \mathscr{F} (\partial_x^\alpha T) = (i\xi)^\alpha \mathscr{F}T, \qquad \mathscr{F} (x^\alpha T) = (-i\partial_\xi)^\alpha \mathscr{F}T, \]
where multiplication by monomials and derivatives are intended in distributional sense. Moreover, the inversion formula and Riemann-Lebsegue theorem hold in $\mathcal S'(\R^n)$.
\begin{theorem}\label{thm:inversion}
Let $T\in\mathcal S'(\R^n)$, with $\mathscr{F}T\in L^1(\R^n)$. Then~$T\in\mathcal C_0(\R^n)$ and
\[ T(x) = (2\pi)^{-n}\,\int_{\R^n} e^{ix\cdot\xi} (\mathscr{F}T)(\xi)\,d\xi. \]
\end{theorem}
We are also interested in an equivalent definition of $H^1(\R^n)$ and its norm, that involves the Fourier transform, and that is interest for the application to the partial differential equations.
\begin{definition}
For $j=1,\ldots,n$, the Riesz transform $R_j$ are defined for any $\varphi\in\mathcal S(\R^n)$ and for suitable $c_n>0$ as
\[ R_j\varphi = c_n\,\PV \frac{x_j}{|x|^{n+1}}\ast \varphi; \quad\text{then,}\quad \mathscr{F}(R_j\varphi)=i\frac{\xi_j}\xii\,\mathscr{F}\varphi. \]
\end{definition}
It is easy to check that Riesz transforms may be extended as linear bounded functionals on $L^2(\R^n)$. By the Theorem of Riesz-Kolmogorov, they can be extended to $L^p(\R^n)$ as well, for $p\in(1,\infty)$. The Riesz transforms cannot be extended to a bounded linear functional on $L^1(\R^n)$, but they can be extended to bounded linear functionals from $L^1$ to $L^{1,\infty}$.
\begin{theorem}\label{thm:HardyR}
Let $f\in L^1(\R^n)$. Then $f\in H^1(\R^n)$ if, and only if, $R_jf\in L^1$ for any $j=1,\ldots,n$, where $R_j$ are the Riesz transforms. Moreover,
\[ \|f\|_{H^1} \sim \|f\|_{L^1}+\sum_{j=0}^n \|R_jf\|_{L^1}. \]
\end{theorem}
As a consequence of Theorem~\ref{thm:HardyR} and Inversion Theorem~\ref{thm:inversion} it follows that if $f\in H^1(\R^n)$, then $\mathscr{F}(f)=0$, that is, $f$ has zero integral. Also, the Riesz transforms are bounded operators on $H^1$.

We recall the basics of the multipliers theory introduced in H\"ormander's paper~\cite{Hormander-1960}. A tempered distribution $m\in\mathcal S'(\R^n)$ is said to be a multiplier from $L^p$ to $L^q$, $m\in M(L^p,L^q)$ if, and only if, $\mathscr{F}^{-1}m\ast$ is a bounded linear operator from $L^p$ to $L^q$. Similarly, if we replace $L^1$ by $H^1$ or $L^\infty$ by $\BMO$. By Plancherel's theorem $M_2^2=L^\infty$ and since $M_{q'}^{p'}=M_p^q$, by interpolations, it follows that
\[ M(L^1,L^1)\subset M(L^p,L^p) \subset M(L^2,L^2)=L^\infty, \quad p\in[1,\infty]. \]
Moreover,
\[ M(L^1,H^1) \subset M(L^1,L^1)\subset M(H^1,H^1)=M(H^1,L^1)\subset M(L^p,L^p),\quad p\in(1,\infty).\]
When the first space is $L^1$, the Young theorem is optimal, in the sense that:
\[\begin{split}
M(L^1,L^q)&=\mathscr{F}(L^q), \qquad q\in(1,\infty],\\
M(L^1,L^1)&=\mathscr{F}(\{\text{bounded measures}\}),\\
M(L^1,H^1)&=\mathscr{F}(H^1).
\end{split}\]

\section{Proof of Theorem~\ref{thm:main}}\label{sec:proofmain}

\begin{proof}
We assume with no restriction that $\varepsilon\in(0,1)$. Assume that $n$ is even and write
\[ (-\Delta)^{\frac{n}2} m  = \mathscr{F} (|x|^nT). \]
Since $(-\Delta)^{\frac{n}2} m\in L^1$, by Theorem~\ref{thm:inversion} for tempered distributions, there exists $g\in\mathcal C_0(\R^n)$ such that
\[ \begin{split}
& |x|^nT = g\qquad \text{in $\mathcal S'(\R^n)$, that is,}\\
& \<T,|x|^n \varphi\> = \int_{\R^n} g \varphi dx, \quad \varphi\in\mathcal S(\R^n),
\end{split} \]
and the Fourier inversion formula holds:
\[ \begin{split}
g
    & = (2\pi)^{-n}\,(I_0(x)+I_1(x)), \\
I_0(x)
    & =\int_{B(0,|x|^{-1})} e^{ix\cdot\xi} ((-\Delta)^{\frac{n}2} m)(\xi)\,d\xi, \\
I_1(x)
    & =\int_{B^c(0,|x|^{-1})} e^{ix\cdot\xi} ((-\Delta)^{\frac{n}2} m)(\xi)\,d\xi.
\end{split} \]
We now prove that
\[ f = \frac{g}{|x|^n} \in L^1(\R^n), \]
obtaining suitable estimates for $g(x)$ (clearly, $f\in\mathcal C(\R^n\setminus\{0\})$). In particular,
\[ \begin{split}
|I_0(x)|
    & \leq C\int_{B(0,|x|^{-1})} \xii^{\varepsilon-n}\,\<\xi\>^{-2\varepsilon}\,d\xi\leq C'\, |x|^{-\varepsilon}, \qquad |x|\ge1,\\
|I_1(x)|
    & \leq C\int_{B^c(0,|x|^{-1})} \xii^{\varepsilon-n}\,\<\xi\>^{-2\varepsilon}\,d\xi\leq C'\, |x|^\varepsilon, \qquad |x|\le1.
\end{split} \]
By performing one step of integration by parts, we get
\[ \begin{split}
I_0(x)
    & = -i \sum_{j=1}^n x_j \int_{B(0,|x|^{-1})} e^{ix\cdot\xi} (\partial_{\xi_j}(-\Delta)^{\frac{n}2-1}m)(\xi)\,d\xi \\
    & \qquad + \sum_{j=1}^n \int_{\partial B(0,|x|^{-1})} e^{ix\cdot\xi} (\partial_{\xi_j}(-\Delta)^{\frac{n}2-1}m)(\xi)\,\nu_j\,d\sigma(\xi);\\
|I_0(x)|
    & \leq C\,|x|\,\int_{B(0,|x|^{-1})} \xii^{\varepsilon-(n-1)}\,\<\xi\>^{-2\varepsilon}\,d\xi \\
    & \qquad + C\,\int_{\partial B(0,|x|^{-1})} \xii^{\varepsilon-(n-1)}\,\<\xi\>^{-2\varepsilon}\,d\sigma(\xi) \\
    & \leq C' |x|^\varepsilon \qquad |x|\le1.
\end{split} \]
Similarly,
\[ \begin{split}
I_1(x)
    & = i|x|^{-2}\sum_{j=1}^n x_j \left(\int_{B(0,|x|^{-1})} e^{ix\cdot\xi} (\partial_{\xi_j}(-\Delta)^{\frac{n}2} m)(\xi)\,d\xi \right.\\
    & \qquad\qquad\qquad\qquad \left. -\int_{\partial B(0,|x|^{-1})} e^{ix\cdot\xi} ((-\Delta)^{\frac{n}2} m)(\xi)\,\nu_j\,d\sigma(\xi) \right); \\
|I_1(x)|
    & \leq C|x|^{-1} \int_{B^c(0,|x|^{-1})} \xii^{\varepsilon-n-1}\,\<\xi\>^{-2\varepsilon}\,d\xi \\
    & \qquad + C|x|^{-1}\int_{\partial B(0,|x|^{-1})} \xii^{\varepsilon-n}\,\<\xi\>^{-2\varepsilon}\,d\sigma(\xi) \\
    & \leq C\,|x|^{-\varepsilon}, \qquad |x|\ge1.
\end{split} \]
Therefore, we proved that $|g(x)|\leq C_1|x|^\varepsilon\<x\>^{-2\varepsilon}$; hence,
\begin{equation}\label{eq:f}
f\in L^1(\R^n), \qquad |f(x)|\leq C_1|x|^{\varepsilon-n}\<x\>^{-2\varepsilon}.
\end{equation}
By the fact that
\[ |x|^n (T-f) =0, \qquad \text{in $\mathcal S'(\R^n)$,} \]
there exists $T_0$, supported at the origin, such that $T = f+T_0$. Since $T_0$ is supported at the origin, $T_0=\sum_{|\alpha|\leq \kappa} a_\alpha \partial_x^\alpha \delta$ for some $\{a_\alpha\}$ and $\kappa$. But~\eqref{eq:mder} implies that $a_\alpha=0$ for any $|\alpha|\geq n-1$. Indeed, $\mathscr{F} (\partial_x^\alpha\delta)=(i\xi)^\alpha$ and $\partial_\xi^{\beta}\xi^\alpha$ is a nonzero monomial of degree $|\alpha|-(n-1)$ for any $\beta\leq\alpha$ (that is, $\beta_j\leq\alpha_j$ for $j=1,\ldots,n$). In particular, if $|\alpha|\geq n-1$, choosing $|\beta|=n-1$, \eqref{eq:mder} holds only if $a_\alpha=0$. Therefore, we have obtained~\eqref{eq:mT0}.

Now let $m(0)=a_0$; then, $\mathscr{F}f(0)=0$, that is, the integral of $f$ is zero. Due to the fact that $f\in L^p(\R^n)$ for any $1\leq p<n/(n-\varepsilon)$, we may apply Lemma~\ref{lem:H1} and obtain $f\in H^1(\R^n)$.

\bigskip

We now consider the case of $n$ odd. In this case, we write
\[ (-\Delta)^{\frac{n-1}2} \nabla m  = \mathscr{F} (ix |x|^{n-1}T). \]
Then, for any $j=1,\ldots,n$, $x_j|x|^{n-1}T =g_j$ in $\mathcal S'(\R^n)$ with $g_j\in\mathcal C_0(\R^n)$. Then $|x|^{n-1}T=g_j/x_j$ on $U_j=\{x\in\R^n: \ x_j\neq0\}$ for any $j=1,\ldots,n$. Moreover,
\[ \frac{g_j}{x_j}=\frac{g_k}{x_k} \qquad x\in U_j\cap U_k. \]
We may define $g$ piecewisely on $\R^n\setminus\{0\}$ as $g=g_j/x_j$ when $x_j\neq0$. Let $f=g/|x|^{n-1}$. Then $f$ is as in~\eqref{eq:f}.

Since $|x|^{n-1}(T-f)=0$ in $\mathcal S'(\R^n)$, we find once again $T=f+T_0$, and the rest of the proof follows as for the case of $n$ even.

This concludes the proof.
\end{proof}

\section{Proof of Lemmas~\ref{lem:H1} and~\ref{lem:H1large}}\label{sec:H1}

We first prove Lemma~\ref{lem:H1}.
\begin{proof}
Let $M\gg1$ be such that
\[ |f(x)|\leq |x|^{-n-\varepsilon},\qquad |x|\geq M. \]
We decompose $f=f_0+f_1$, where $f_0$ and $f_1$ have zero integral and:
\begin{itemize}
\item $f_0$ has compact support;
\item $f_1=0$ in the ball $B(0,M)$.
\end{itemize}
For instance, let $\chi,\eta\in \mathcal C_c(\R^n)$ be such that $\chi=1$ and $\eta=0$ in $B(0,M)$, with $\int \eta=1$. Setting
\[ f_0 = \chi f - a\eta, \qquad a = \int \chi f,\]
both $f_0$ and $f_1=f-f_0$ verify the desired properties.

We notice that $f_0\in L^p(\R^n)$, with compact support, so that $f_0\in H^1$ (see, for instance, \cite[\textsection 1.2.4]{stein1993harmonic}). On the other hand, $f_1\in H^1$ by Lemma~\ref{lem:H1large} and this concludes the proof.
\end{proof}
We now prove Lemma~\ref{lem:H1large}.
\begin{proof}
We use the atomic decomposition. Let
\[ g_k = \psi_k \left( f - \frac{\int f\psi_k}{\int \psi_k}\right),  \]
with $\psi_k(x)=\psi(k^{-1}x)$, $\psi\in\mathcal C_c^\infty$, supported in $B(0,1)$ such that $\psi=1$ on $B(0,1/2)$. In particular,
\[ \int\psi_k = k^n\int\psi. \]
We notice that $g_k$ has zero integral and that
\[ \int f\psi_k = - \int f(1-\psi_k) \]
thanks to the fact that $f$ has zero integral. In particular,
\[ |\int f\psi_k| = |\int f(1-\psi_k)| \leq C \int_{|x|\geq k/2} |x|^{-n-\varepsilon}dx \sim k^{-\varepsilon}, \]
so that
\[ \frac{|\int f\psi_k|}{|\int \psi_k|} \leq Ck^{-n-\varepsilon}. \]
Then $g_{2^{j+1}}-g_{2^j}$ has zero integral, it is supported in $B(0,2^{j+1})\setminus B(0,2^{j-1})$, and
\[ |g_{2^{j+1}}-g_{2^j}| \leq C(2^j)^{-n-\varepsilon}\sim |B(0,2^{j+1})|^{-1}\,2^{-j\varepsilon}. \]
Due to $f(x)=0$ for $|x|\leq M$,
\[ f = \sum_{j\geq j_0} (g_{2^{j+1}}-g_{2^j}) \]
with $j_0\sim \log_2M$. Since
\[ g_{2^{j+1}}-g_{2^j} = \lambda_j h_j, \qquad |\lambda_j|\leq C2^{-j\varepsilon} \]
with $h_j$ atoms, the proof follows from
\[ \|f\|_{H^1}\leq \sum_{j\geq j_0} |\lambda_j|\leq C\,\sum_{j\geq j_0}2^{-j\varepsilon}\sim 2^{-j_0\varepsilon}\sim M^{-\varepsilon}. \]
\end{proof}

\bibliographystyle{plain}
\bibliography{ref.bib}

\end{document}